\documentclass[a4paper]{article}
\usepackage{amsmath}
\numberwithin{equation}{section}
\usepackage{amssymb,amsfonts,amsthm}
\usepackage[margin=1in]{geometry}
\usepackage[shortlabels]{enumitem}
\usepackage{hyperref} 
\usepackage{xcolor}
\usepackage{url}

\title{Perfect Quantum Approximate Strategies for Imitation Games}
\author{Hao Liang, Tianshi Yu, Lihong Zhi}

\begin{document}

\maketitle
\newtheorem{defi}{Definition}[section]
\newtheorem{thm}{Theorem}[section]
\newtheorem{cor}[thm]{Corollary}
\newtheorem{prop}[thm]{Proposition}
\newtheorem{lem}{Lemma}[section]
\newtheorem{claim}{Claim}
\newtheorem{remark}{Remark}
\newtheorem{Lemma}{Lemma}
\newtheorem{Example}{Example}
\newcommand{\bN} { {\mathbb{N}}}   
\newcommand{\bC} { {\mathbb{C}}}  
\newcommand{\bQ} { {\mathbb{Q}}}   
\newcommand{\bZ} { {\mathbb{Z}}}   
\newcommand{\bR} { {\mathbb{R}}}   
\newcommand{\bF} { {\mathbb{F}}}
\newcommand{\bH} { {\mathbb{H}}}
\newcommand{\bK} { {\mathbb{K}}}
\newcommand{\op} { {\mathrm{op}}}
\newcommand{\alg} { {\mathrm{alg}}}
\newcommand*\abs[1]{\left\lvert#1\right\rvert} 
\newcommand*\norm[1]{\left\lVert#1\right\rVert} 

\let\bar=\overline
\let\leq=\leqslant
\let\geq=\geqslant

\begin{abstract}
	We prove that an imitation game has a perfect quantum approximate strategy if and only if there exists a bi-tracial state on the minimal tensor product of two universal C${^*}$-algebras, which induces the perfect correlation. Moreover, we are trying to relate imitation games to the minimal tensor product of two universal C${^*}$-algebras	and demonstrate that an imitation game has a perfect quantum approximate strategy if and only if there exist a von Neumann algebra and an amenable tracial state on it, such that the perfect correlation can be induced by the tracial state. However, We encountered some difficulties regarding continuity in the proof process. In section 2 we get some results for special cases, and in section 3 we list our problems.




\end{abstract}

\section{Preliminaries}

Let $ X,Y,A,B $ be finite sets. 
A {\it non-local game} $ \mathcal{G}=(X,Y,A,B,\lambda) $ involves a verifier and two players, Alice and Bob. We think of $ X,Y $ as question sets for Alice and Bob, and $ A,B $ as answer sets for them, respectively.
After receiving an answer from each player, the verifier  evaluates the  {\it scoring function}
\begin{equation}\label{eq2.6}
\lambda:~X\times Y\times A\times B\longrightarrow \{0,1\}
\end{equation}
If $ \lambda(x,y,a,b)=1$, we say Alice and Bob win; otherwise, they lose the game.
Alice and Bob know the sets $ X, Y, A, B $ and the scoring function $\lambda$,
but they can't communicate during the game.  
Alice and Bob can make some arrangements before the game starts. 

A {\it correlation}  for the game $\mathcal{G}$ is a conditional probability density $p=\left(p\left(a,b\mid x,y\right)\right)_{X\times Y\times A\times B}$, which satisfies $\sum_{(a,b)\in A\times B}p(a,b\mid x,y)=1,~\forall x\in X,~y\in Y.$
 We say a correlation $ p$ 
is a {\it perfect strategy} for $ \mathcal{G} =(X, Y, A, B, \lambda)$ if 
\begin{equation}\label{eq2.7}
\lambda(x,y,a,b)=0\Longrightarrow p(a,b\mid x,y)=0.
\end{equation}

In this paper, we focused on characterizing the perfect quantum approximate strategies of imitation games. 

\begin{defi}\cite{lupini2020perfect}\label{D2.2}
	A non-local game $\mathcal{G}=(X, Y, A, B, \lambda)$ is called an imitation game if
	\begin{enumerate}[(a)]
		\item\label{ca} for any $x \in X$ and $a, a^{\prime} \in A$ with $a \neq a^{\prime}$, there exists $y \in Y$ such that
		\begin{equation*}
		    \sum_{b \in B} \lambda(a, b \mid x, y) \lambda\left(a^{\prime}, b \mid x, y\right)=0,~and
		\end{equation*}

		\item\label{cb} for any $y \in Y$ and $b, b^{\prime} \in B$ with $b \neq b^{\prime}$, there exists $x \in X$ such that
		\begin{equation*}
		    \sum_{a \in A} \lambda(a, b \mid x, y) \lambda\left(a, b^{\prime} \mid x, y\right)=0 .
		\end{equation*}
	\end{enumerate}
\end{defi}

Conditions \ref{ca} and \ref{cb} in Definition \ref{D2.2} imply that, in a perfect strategy for $\mathcal{G}$, Alice's and Bob's answers are mutually determined.

Now, we introduce the concept of correlation. It's well-known that given a non-local game $\mathcal{G}$, the strategy of the players and the correlation for them are mutually determined.

\begin{defi}
A non-signalling correlation is a function $ p:X\times Y\times A\times B\rightarrow[0,1] $ which satisfies:
\begin{align*}
&\sum_{b\in B}p(a,b\mid x,y)=\sum_{b\in B}p(a,b\mid x,y^{\prime}),~\forall x\in X,~y,y^{\prime}\in Y,~a\in A\\
&\sum_{a\in A}p(a,b\mid x,y)=\sum_{a\in A}p(a,b\mid x^{\prime},y),~\forall x,x^{\prime}\in X,~y\in Y,~a\in A\\
&\sum_{(a,b)\in A\times B}p(a,b\mid x,y)=1,~\forall x\in X,~y\in Y.
\end{align*}
We denote the set of all non-signaling correlations by $ \mathcal{C}_{ns}(X,Y,A,B)$. Furthermore, for a scoring function $\lambda$, we denote the set of perfect non-signaling correlations as $ \mathcal{C}_{ns}(\mathcal{G})$. 

 \end{defi}
 

Next, we define quantum (resp. quantum approximate, commuting-operator) correlations. We use Dirac's notation. Recall that we can use PVM (projection-operator valued measure) instead of POVM (positive operator valued measure) to describe the definition\cite{fritz2014operator}. We denote the algebra of all bounded linear operators on Hilbert space $\mathcal{H}$ as $\mathcal{B}(\mathcal{H})$.

\begin{defi}
	For fixed $ X,Y,A,B $, $ p\in\mathcal{C}_{ns}(X,Y,A,B) $ is called a quantum correlation if there exist finite-dimensional Hilbert spaces $ \mathcal{H}_A,\mathcal{H}_B $ and a unit vector $ |\psi\rangle\in\mathcal{H}_A\otimes\mathcal{H}_B $ such that: for every $ x\in X $ and $ y\in Y $, there exist projection-valued measures (PVM)$ \{P_a^x:a\in A\}\subset\mathcal{B}(\mathcal{H}_A) $ and $ \{Q_b^y:b\in B\}\subset\mathcal{B}(\mathcal{H}_B) $ which satisfy
	\begin{equation}\label{eq2.4}
	p(a,b\mid x,y)=\langle\psi|P_a^x\otimes Q_b^y|\psi\rangle.
	\end{equation}
	We denote the set of all quantum correlations by $ \mathcal{C}_q(X,Y,A,B) $.  We define the closure of $ \mathcal{C}_q $ to be $ \mathcal{C}_{qa}(X,Y,A,B) $, which is the set of quantum approximate correlations. 
\end{defi}

More precisely, a correlation $p=\left(p(a,b\mid x,y)\right)\in\mathcal{C}_{qa}$ if and only if: for every $\varepsilon>0$, there exist finite-dimensional Hilbert spaces $ \mathcal{H}_A(\varepsilon),\mathcal{H}_B(\varepsilon) $ and a unit vector $ |\psi_{\varepsilon}\rangle\in\mathcal{H}_A\otimes\mathcal{H}_B $ such that: for every $ x\in X $ and $ y\in Y $, there exist projection-valued measures (PVM)$ \{P_a^x(\varepsilon):a\in A\}\subset\mathcal{B}(\mathcal{H}_A) $ and $ \{Q_b^y(\varepsilon):b\in B\}\subset\mathcal{B}(\mathcal{H}_B) $ which satisfy
\begin{equation}\label{def_qa}
	\left|p(a,b\mid x,y)-\langle\psi_{\varepsilon}|P_a^x(\varepsilon)\otimes Q_b^y(\varepsilon)|\psi_{\varepsilon}\rangle\right|<\varepsilon.
\end{equation}

\begin{defi}	
	The set $ \mathcal{C}_{qc}(X,Y,A,B) $ of quantum commuting correlations consists of the form
	\begin{equation}\label{eq2.5}
	p(a,b\mid x,y)=\langle\psi|P_a^xQ_b^y|\psi\rangle,
	\end{equation}
	where $ \mathcal{H} $ is a Hilbert space, $ |\psi\rangle\in\mathcal{H} $ is a unit vector and $ \{P_a^x:a\in A\},~\{Q_b^y:b\in B\}\subset\mathcal{B}(\mathcal{H}) $ are PVM's which satisfy $ P_a^xQ_b^y=Q_b^yP_a^x $.
\end{defi}

In the case of perfect strategies, we can write 
\begin{equation*}
\mathcal{C}_{\mathrm{x}}(\mathcal{G})=\mathcal{C}_{ns}(\mathcal{G})\cap \mathcal{C}_{\mathrm{x}}(X,Y,A,B),~\mathrm{x}=q,qa,qc,
 \end{equation*}
and we have  
 \begin{equation*}
     \mathcal{C}_q(\mathcal{G})\subseteq\mathcal{C}_{qa}(\mathcal{G})\subseteq\mathcal{C}_{qc}(\mathcal{G})\subseteq\mathcal{C}_{ns}(\mathcal{G}).
 \end{equation*}

In particular, (\ref{def_qa}) shows that a correlation $p=\left(p(a,b\mid x,y)\right)\in\mathcal{C}_{qa}(\mathcal{G})$ iff for every $\varepsilon>0$ and $\lambda(x,y,a,b)=0$, we have
\begin{equation}\label{def_perfect_qa}
	\langle\psi_{\varepsilon}|P_a^x(\varepsilon)\otimes Q_b^y(\varepsilon)|\psi_{\varepsilon}\rangle<\varepsilon
\end{equation}





Slofstra proved that the set of correlations induced by the tensor product structure (finite or infinite dimensional) is not closed  \cite{slofstra2019set}.  Slofstra's proof was soon improved by \cite{dykema2019non} using graph chromatic games, a type of synchronous game. In \cite{musat2020non}  Magdalena Musat and Mikael R$\phi$rdam gave a more explicit proof by the tracial state on C$^*$-algebras and the decomposition of identity operator into a sum of projection. 
In \cite{ji2020mip}, Ji et al. provided a breakthrough result that negatively resolved Connes' embedding conjecture. They showed that there are indeed quantum correlations that can be described using commuting operators on a single Hilbert space but cannot be approximated by a sequence of tensor product finite-dimensional Hilbert spaces.

 Quantum correlations can also be expressed as states on C$ ^* $-algebras. We refer the reader to \cite{kadison1997fundamentals,arveson2012invitation,brown88c} for basic knowledge of operator algebra. Here, we list some essential definitions.


\begin{defi}[tracial state]
	Let $ \mathcal{A} $ be a unital C$ ^* $-algebra, we say $ \phi $ is a state of $ \mathcal{A} $ if $ \phi:\mathcal{A}\rightarrow\bC $ is a positive linear functional and $ \phi(1)=1 $. A state $ \tau:\mathcal{A}\rightarrow\bC $ is called a tracial state if $ \tau(ab)=\tau(ba),~\forall a,b\in\mathcal{A} $.
\end{defi}

Given two C$ ^* $-algebras $ \mathcal{A} $ and $ \mathcal{B} $, their algebraic tensor product $ \mathcal{A}\otimes_{\alg}\mathcal{B} $ is a $ * $-algebra. We can define different norms on it to make $ \mathcal{A}\otimes_{\alg}\mathcal{B} $ into a C$ ^* $-algebra. One norm on $ \mathcal{A}\otimes_{\alg}\mathcal{B} $ is 
\begin{equation}\label{eq2.9}
\norm{a}_{\max}:=\underset{\pi}{\sup}\{\norm{\pi(a)} \mid \pi:\mathcal{A}\otimes_{\alg}\mathcal{B}\rightarrow\mathcal{B}(\mathcal{H})~ \text{is a representation of~}\mathcal{A}\otimes_{\alg}\mathcal{B}\},
\end{equation}
where $ \mathcal{H} $ is a Hilbert space. After completion we get a C$ ^* $-algebra $ \mathcal{A}\otimes_{\max}\mathcal{B} $, called max-tensor product. We can also define another norm as follows:
\begin{equation}\label{eq2.10}
\norm{\sum_{i=1}^{n}a_i\otimes b_i}_{\min}:=\norm{\sum_{i=1}^{n}\pi_1(a_i)\otimes\pi_2(b_i)}_{\mathcal{B}(\mathcal{H}_1\otimes\mathcal{H}_2)},
\end{equation}
where $ \pi_i:\mathcal{A}\rightarrow\mathcal{B}(\mathcal{H}_i),~i=1,2 $ are faithful representations. This norm is well defined and independent of the choice of faithful representations \cite{brown88c}, and after completion, we can get another C$ ^* $-algebra $ \mathcal{A}\otimes_{\min}\mathcal{B} $.

\begin{defi}[amenable]\cite{brown88c,kim2018synchronous}
	Let $\mathcal{A} \subseteq B(\mathcal{H})$ be a $C^*$-algebra. A tracial state $\tau$ on $\mathcal{A}$ is called amenable if there is a state $\rho$ on $B(\mathcal{H})$ such that $\rho\big|_{\mathcal{A}}=\tau$ and $\rho\left(u T u^*\right)=\rho(T)$ for all $T \in B(\mathcal{H})$ and all unitaries $u \in \mathcal{A}$.
\end{defi}

The amenable tracial state has several equivalent characterizations, see \cite[Theorem 6.2.7]{brown88c}.  Here, we only list the equivalent conditions that we require. Let $ \mathcal{A} $ be a C$ ^* $-algebra. Recall that the opposite C$ ^* $-algebra $ \mathcal{A}^{\op} $ is the C*-algebra that one obtains by preserving the Banach space structure on $ \mathcal{A} $ but defining a product $ "\bullet" $ by $ a\bullet b=ba $.

\begin{prop}\label{P2.2}\cite{brown88c}
Let $\mathcal{A}$ be a unital C$^*$-algebra with tracial state $\tau$. The following are equivalent:
\begin{enumerate}[(1)]
    \item $\tau$ is amenable;
    \item the linear map $\sigma:\mathcal{A}\otimes_{\alg}\mathcal{A}^{\op}\rightarrow\bC,~a\otimes b\mapsto \tau(ab)$ is bounded with respect to the minimal tensor product norm.
\end{enumerate}
\end{prop}

We denote $ \mathcal{A}(X,A) $ to be {\it{the universal  C${^*}$-algebra}} generated by projections $ e_a^x$ satisfying  $ \sum_{a\in A}e_a^x=1 $ for all $ x\in X $, and  $ \mathcal{A}(Y,B) $ to be the universal  C${^*}$-algebra generated by projections $ f_b^y $ satisfying $ \sum_{b\in B}f_b^y=1 $ for all $ y\in Y $.
What's more, we denote the $\bC$-subalgebra generated by all the $e_a^x$ as $\bC\langle X,A\rangle$, and similarly for $\bC\langle Y,B\rangle$. It's easy to deduce that both $\mathcal{A}(X,A)$ and $\mathcal{A}(Y,B)$ are separable, and this fact implies that $\mathcal{A}(X,A)\otimes_{\min}\mathcal{A}(Y,B)$ is also separable with respect to the minimal norm of the tensor product.

In \cite[Propsition 3.4]{fritz2012tsirelson} (see also \cite[Theorem 2.10]{paulsen2015quantum}), Fritz characterized quantum correlations in terms of C$^*$-algebra tensor products. Our main result is based on the following proposition: 
\begin{prop}\label{P2.1}\cite{fritz2012tsirelson}
	Suppose $ p=(p(a,b\mid x,y))_{X\times Y\times A\times B}\in\mathcal{C}_{ns}(X,Y,A,B) $, then the following statements are equivalent:
	\begin{enumerate}[(1)]
	\item $ p\in\mathcal{C}_{qa}(X,Y,A,B) $;
	\item There exists a state $ \tau $ on $ \mathcal{A}(X,A)\otimes_{\min}\mathcal{A}(Y,B) $ such that
	\begin{equation*}
	p(a,b\mid x,y)=\tau(e_a^x\otimes f_b^y).
	\end{equation*}
	\end{enumerate}
\end{prop}

In the case of mirror games, which is a special case of the imitation game, Lupini et al.  extended   in \cite[Theorem III.6]{kim2018synchronous} for synchronous games and got the following characterization for perfect quantum approximate correlations  \cite[Theore 6.4]{lupini2020perfect}:
\begin{prop}\label{Mirror}\cite{lupini2020perfect} 
	Let $\mathcal{G}$ be a mirror game. Suppose $ p=(p(a,b\mid x,y))_{X\times Y\times A\times B}\in\mathcal{C}_{ns}(\mathcal{G}) $, then the following statements are equivalent:
	\begin{enumerate}[(1)]
	\item $ p\in\mathcal{C}_{qa}(\mathcal{G}) $;
	\item There exist an amenable tracial state $ \tau:\mathcal{A}(X,A)\rightarrow\bC $ and a unital *-homomorphism 
    $\rho:\mathcal{A}(Y,B)\rightarrow\mathcal{A}(X,A)$ with 
    $\rho(\mathcal{S}_{Y,B})\subseteq\mathcal{S}_{X,A}$ 
    \footnote{There exists a typo in \cite{lupini2020perfect} where the order of $\mathcal{S}_{X,A}$ and $\mathcal{S}_{Y,B}$ is reversed. Here, we provide the correct form.}
    such that
	\begin{equation*}
	p(a,b\mid x,y)=\tau(e_a^x\rho(f_b^y)),~x\in X,y\in Y,a\in A,b\in B,
	\end{equation*}
	\end{enumerate}
    where $\mathcal{S}_{X,A}=\mathrm{span}\{e_a^x:~x\in X,A\in A\}$ and similarly for $\mathcal{S}_{Y,B}$.
\end{prop}

This characterization uses an amenable tracial state on a smaller algebra $\mathcal{A}(X,A)$ indexed by one player. Our aim is to generalize Proposition \ref{Mirror} to imitation games, i.e., using a smaller algebra "generated" by one player's projections to characterize the perfect quantum approximate strategy. In fact, the von Neumann algebra $\mathcal{U}$, which will be introduced in Theorem \ref{Thm3.1}, plays such a role. However, only for very special case we can get the amenable tracial state. We ask whether there is a way to generalize this theorem to the general case, or if there are other characterizations for the general case.



\section{Main Result}

Firstly we present a strengthened version of Proposition \ref{P2.1}.

\begin{thm}\label{independent}
	Let $ \mathcal{G}=(X,Y,A,B,\lambda) $ be an imitation game, and $ p\in\mathcal{C}_{ns}(\mathcal{G}) $ be a perfect non-signalling correlation. The following statements are equivalent:
	\begin{enumerate}[(1)]
		\item $p\in\mathcal{C}_{qa}(\mathcal{G})$;
		\item There exists a state $\tau$ on $\mathcal{A}(X,A)\otimes_{\min}\mathcal{A}(Y,B)$ such that 
		\begin{equation*}
		p(a,b\mid x,y)=\tau(e_a^x\otimes f_b^y),
		\end{equation*}
		and $\tau\big|_{\mathcal{A}(X,A)\otimes\mathbf{1}},~\tau\big|_{\mathbf{1}\otimes\mathcal{A}(Y,B)}$ are both tracial.
	\end{enumerate}
\end{thm}
\begin{proof}
	(2)$\Rightarrow$(1) can be directly obtained from Proposition \ref{P2.1}. Now we prove the opposite direction.
	
	Since $p\in\mathcal{C}_{qa}(\mathcal{G})$, by (\ref{def_qa}) we know that for every $\varepsilon>0$, there exist finite dimensional Hilbert space $\mathcal{H}_1(\varepsilon),\mathcal{H}_2(\varepsilon)$, a unit vector $\psi_{\varepsilon}$, PVM $\{P_a^x(\varepsilon): a\in A\}\subset\mathcal{B}(\mathcal{H}_1(\varepsilon))$ for every $x\in X$, $\{Q_b^y(\varepsilon): b\in B\}\subset\mathcal{B}(\mathcal{H}_2(\varepsilon))$ for every $y\in Y$ such that 
	\begin{equation*}
	\abs{p(a,b\mid x,y)-\langle\psi_{\varepsilon}|P_a^x(\varepsilon)\otimes Q_b^y(\varepsilon)|\psi_{\varepsilon}\rangle}<\varepsilon.
	\end{equation*} 
	For any $x\in X,~y\in Y$ and $b\in B$, define
	\begin{equation*}
	{}^x\Pi_b^y(\varepsilon)=\sum_{a \in A, \lambda(x, y, a, b)=1}P_a^x(\varepsilon),
	\end{equation*}
	then ${}^x\Pi_b^y(\varepsilon)$ is also a projection in $\mathcal{B}(\mathcal{H}_1(\varepsilon))$. Firstly by (\ref{def_perfect_qa}) we have
	\begin{align*}
	&\quad	\langle\psi_{\varepsilon}|(1-{}^x\Pi_b^y(\varepsilon))\otimes Q_b^y(\varepsilon)|\psi_{\varepsilon}\rangle\\
	&=\sum_{a \in A, \lambda(x, y, a, b)=0}\langle\psi_{\varepsilon}|P_a^x(\varepsilon\otimes Q_b^y(\varepsilon))|\psi_{\varepsilon}\rangle\\
	&=\sum_{a \in A, \lambda(x, y, a, b)=0}p(a,b\mid x,y)\\
	&<|A|\varepsilon.
	\end{align*}
	Since $(1-{}^x\Pi_b^y(\varepsilon))\otimes Q_b^y(\varepsilon)$ is an idempotent element in $\mathcal{B}(\mathcal{H}_1(\varepsilon)\otimes\mathcal{H}_2(\varepsilon))$, we know that
	\begin{equation*}
	\norm{(1-{}^x\Pi_b^y(\varepsilon))\otimes Q_b^y(\varepsilon)|\psi_{\varepsilon}\rangle}=
	\langle\psi_{\varepsilon}|(1-{}^x\Pi_b^y(\varepsilon))\otimes Q_b^y(\varepsilon)|\psi_{\varepsilon}\rangle^{\frac{1}{2}}<M_1\varepsilon^{\frac{1}{2}}.
	\end{equation*}
	By the assumption of imitation games, for $b \neq b^{\prime}$ and $y \in Y$ there exists $x \in X$ such that 
	\begin{equation*}
		\sum_{a \in A} \lambda(x, y, a, b) \lambda\left(x, y, a, b^{\prime}\right)=0.
	\end{equation*}
	For such a choice of $x, y, b, b^{\prime}$, we have
	\begin{equation*}
		\left\langle\psi_{\varepsilon}\left|{}^x\Pi_b^y(\varepsilon)\otimes Q_{b^{\prime}}^{y}(\varepsilon)\right| \psi_{\varepsilon}\right\rangle=\sum_{a \in A, \lambda(x, y, a, b)=1}\left\langle\psi_{\varepsilon}\left|P_a^x(\varepsilon) Q_{b^{\prime}}^y(\varepsilon)\right| \psi_{\varepsilon}\right\rangle=\sum_{a \in A, \lambda(x, y, a, b)=1}p(a,b^{\prime}\mid x,y).
	\end{equation*}
	By the definition of imitation game we obtain that $p(a,b^{\prime}\mid x,y)<\varepsilon$ when $\lambda(x, y, a, b)=1$, so
	\begin{equation*}
	\left\langle\psi_{\varepsilon}\left|{}^x\Pi_b^y(\varepsilon)\otimes Q_{b^{\prime}}^{y}(\varepsilon)\right| \psi_{\varepsilon}\right\rangle<|A|\varepsilon,
	\end{equation*}
	i.e
	\begin{equation*}
	\|{}^x\Pi_b^y(\varepsilon)\otimes Q_{b^{\prime}}^{y}(\varepsilon)| \psi_{\varepsilon}\rangle\|<M_2\varepsilon^{\frac{1}{2}}.
	\end{equation*}
	Define
	\begin{equation*}
	\Pi_b^y(\varepsilon)=\underset{x\in X}{\bigwedge}{}^x\Pi_b^y,
	\end{equation*}
	by De Morgan's laws we have
	\begin{equation*}
	1-\Pi_b^y(\varepsilon)=\underset{x\in X}{\bigvee}(1-{}^x\Pi_b^y(\varepsilon)).
	\end{equation*}
	Thus we obtain
	\begin{align*}
		\quad & \norm{(1-\Pi_b^y(\varepsilon))\otimes Q_b^y(\varepsilon)|\psi_{\varepsilon}\rangle} \\
		= & \norm{(\underset{x\in X}{\bigvee}(1-{}^x\Pi_b^y(\varepsilon)))\otimes Q_b^y(\varepsilon)|\psi_{\varepsilon}\rangle}\\
		\leq&\sum_{x\in X}\norm{(1-{}^x\Pi_b^y(\varepsilon))\otimes Q_b^y(\varepsilon)|\psi_{\varepsilon}\rangle}\\
		<& M_3\varepsilon^{\frac{1}{2}}.
	\end{align*}
	For $b^{\prime}\neq b\in B$, we also have
	\begin{equation*}
	\norm{\Pi_b^y(\varepsilon)\otimes Q_{b^{\prime}}^{y}(\varepsilon)| \psi_{\varepsilon}\rangle}\leq\|{}^x\Pi_b^y(\varepsilon)\otimes Q_{b^{\prime}}^{y}(\varepsilon)| \psi_{\varepsilon}\rangle\|<M_4\varepsilon^{\frac{1}{2}}.
	\end{equation*}
	Now we are ready to prove the following inequality:
	\begin{align*}
	\quad&\norm{\Big(1\otimes Q_b^y\left(\varepsilon\right)-\Pi_b^y\left(\varepsilon\right)\otimes 1\Big)|\psi_{\varepsilon}\rangle}\\
	=&\norm{\Big(1\otimes Q_b^y\left(\varepsilon\right)-\Pi_b^y\left(\varepsilon\right)\otimes \sum_{b^{\prime}\in B}Q_{b^{\prime}}^{y}(\varepsilon)\Big)|\psi_{\varepsilon}\rangle}\\
	\leq&\norm{(1-\Pi_b^y(\varepsilon))\otimes Q_b^y(\varepsilon)|\psi_{\varepsilon}\rangle}+\sum_{b^{\prime}\neq b}\norm{\Pi_b^y(\varepsilon)\otimes Q_{b^{\prime}}^{y}(\varepsilon)| \psi_{\varepsilon}\rangle}\\
	<&M_5\varepsilon^{\frac{1}{2}}
	\end{align*}
	Therefore, for every $y_1,y_2\in Y$ and $b_1,b_2\in B$, we have
	\begin{align*}
	\quad&\abs{\langle\psi_{\varepsilon}|\Big((1\otimes Q_{b_1}^{y_1}(\varepsilon))(1\otimes Q_{b_2}^{y_2}(\varepsilon))-(1\otimes Q_{b_1}^{y_1}(\varepsilon))(\Pi_{b_2}^{y_2}(\varepsilon)\otimes 1)\Big)|\psi_{\varepsilon}\rangle}\\
	=&\abs{\langle (1\otimes Q_{b_1}^{y_1}(\varepsilon))\psi_{\varepsilon}\mid\Big(1\otimes Q_{b_2}^{y_2}(\varepsilon)-\Pi_{b_2}^{y_2}(\varepsilon)\otimes 1\Big)\psi_{\varepsilon} \rangle}\\
	\leq&\norm{(1\otimes Q_{b_1}^{y_1}(\varepsilon))|\psi_{\varepsilon}\rangle}\cdot\norm{\Big(1\otimes Q_{b_2}^{y_2}\left(\varepsilon\right)-\Pi_{b_2}^{y_2}\left(\varepsilon\right)\otimes 1\Big)|\psi_{\varepsilon}\rangle}\\
	<&M_6\varepsilon^{\frac{1}{2}},
	\end{align*}
	
	\begin{equation*}
	\langle\psi_{\varepsilon}|(1\otimes Q_{b_1}^{y_1}(\varepsilon))(\Pi_{b_2}^{y_2}(\varepsilon)\otimes 1)|\psi_{\varepsilon}\rangle=\langle\psi_{\varepsilon}|(\Pi_{b_2}^{y_2}(\varepsilon)\otimes 1)(1\otimes Q_{b_1}^{y_1}(\varepsilon))|\psi_{\varepsilon}\rangle,
	\end{equation*}
	and similarly
	\begin{equation*}
	\abs{\langle\psi_{\varepsilon}|\Big((\Pi_{b_2}^{y_2}(\varepsilon)\otimes 1)(1\otimes Q_{b_1}^{y_1}(\varepsilon))-(1\otimes Q_{b_2}^{y_2}(\varepsilon))(1\otimes Q_{b_1}^{y_1}(\varepsilon))\Big)|\psi_{\varepsilon}\rangle}<M_7\varepsilon^{\frac{1}{2}}.
	\end{equation*}
	The above inequalities imply that
	\begin{equation}\label{appr_trace}
	\abs{\langle\psi_{\varepsilon}|\Big((1\otimes Q_{b_1}^{y_1}(\varepsilon))(1\otimes Q_{b_2}^{y_2}(\varepsilon))-(1\otimes Q_{b_2}^{y_2}(\varepsilon))(1\otimes Q_{b_1}^{y_1}(\varepsilon))\Big)|\psi_{\varepsilon}\rangle}<M_8\varepsilon^{\frac{1}{2}}.
	\end{equation}
	
	Now we can find the state $\tau$. After choosing $\{P_a^x(\varepsilon): a\in A\}$ and $\{Q_b^y(\varepsilon): b\in B\}$, we obtain two *-homomorphism 
	\begin{align*}
	\pi_{\varepsilon,A}:\mathcal{A}(X,A)&\rightarrow\mathcal{B}(\mathcal{H}_1(\varepsilon))\\e_a^x&\mapsto P_a^x(\varepsilon),
	\end{align*} 
	and
	\begin{align*}
		\pi_{\varepsilon,B}:\mathcal{A}(Y,B)&\rightarrow\mathcal{B}(\mathcal{H}_2(\varepsilon))\\f_b^y&\mapsto Q_b^y(\varepsilon).
	\end{align*}
	By \cite[Corollary 3.5.5]{brown88c}, we know
	$\pi_{\varepsilon,A}\otimes_{\min}\pi_{\varepsilon,B}:\mathcal{A}(X,A)\otimes_{\min}\mathcal{A}(Y,B)\rightarrow\mathcal{B}(\mathcal{H}_1(\varepsilon)\otimes\mathcal{H}_2(\varepsilon))$ is a bounded *-homomorphism, thus for every $\varphi>0$,
	\begin{align*}
		\tau_{\varepsilon}:\mathcal{A}(X,A)\otimes_{\min}\mathcal{A}(Y,B)&\rightarrow\bC\\
		e_a^x\otimes f_b^y&\mapsto\langle\psi_{\varepsilon}|P_a^x(\varepsilon)\otimes|\psi_{\varepsilon}\rangle
	\end{align*}
	is a state on $\mathcal{A}(X,A)\otimes_{\min}\mathcal{A}(Y,B)$. Take $\varepsilon$ as $\{\frac{1}{n}:n\in\bN\}$, we get a sequences of states, and by Banach-Alaoglu Theorem we know that there exists a weak* cluster point of $\{\tau_{\frac{1}{n}}\}$, denoted as $\tau$. Since $$\abs{\tau_{\varepsilon}((1\otimes f_{b_1}^{y_1})(1\otimes f_{b_2}^{y_2})-(1\otimes f_{b_2}^{y_2})(1\otimes f_{b_1}^{y_1}))}<M_8\varepsilon^{\frac{1}{2}},$$
	from the arbitrariness of $\varepsilon$ we obtain 
	$$
	\tau((1\otimes f_{b_1}^{y_1})(1\otimes f_{b_2}^{y_2})-(1\otimes f_{b_2}^{y_2})(1\otimes f_{b_1}^{y_1}))=0,
	$$
	and by induction we can prove that $\tau\big|_{\mathbf{1}\otimes\mathcal{A}(Y,B)}$ is tracial. The property that $\tau\big|_{\mathcal{A}(X,A)\otimes 1}$ is tracial can be proved similarly. 
\end{proof}

Next we provide a sufficient condition for $p\in\mathcal{C}_{qa}(\mathcal{G})$ when $\mathcal{G}$ is an imitation game.
\begin{thm}\label{Thm3.2}
	Let $ \mathcal{G}=(X,Y,A,B,\lambda) $ be an imitation game, and $ p\in\mathcal{C}_{ns}(\mathcal{G}) $ be a perfect non-signalling correlation. If there exist a von Neumann algebra $ \mathcal{U} $, a $ * $-homomorphism $ \rho:\bC\langle X,A\rangle\otimes_{\alg}\bC\langle Y,B\rangle\rightarrow \mathcal{U} $ and an amenable tracial state $ \tau $ on $ \mathcal{U} $ such that 
	\begin{equation}\label{eq3.0}
		\tau\circ\rho(e_a^x\otimes f_b^y)=p(a,b\mid x,y),
	\end{equation}
	then we have $p\in\mathcal{C}_{qa}(\mathcal{G})$.
\end{thm}
\begin{proof}
		Since $ \tau $ is amenable on $ \mathcal{U} $, by \cite[Theorem 6.2.7]{brown88c} there is a state 
	\[\sigma:\mathcal{U}\otimes_{\min}\mathcal{U}^{\op}\rightarrow\bC \]
	such that
	\[ \sigma(a\otimes b^{\op})=\tau(ab).\]
	Using $ \rho $,  we can also define the following $*$-homomorphisms 
	\begin{align*}
		\rho_1:\bC\langle X,A\rangle&\rightarrow\mathcal{U}\\
		e_{a_1}^{x_1}\cdots e_{a_n}^{x_n}&\mapsto\rho(e_{a_1}^{x_1}\cdots e_{a_n}^{x_n}\otimes 1)
	\end{align*}
	and
	\begin{align*}
		\rho_2:\bC\langle Y,B\rangle&\rightarrow\mathcal{U}^{\op}\\
		f_{b_1}^{y_1}\cdots f_{b_n}^{y_n}&\mapsto\rho(1\otimes f_{b_n}^{y_n}\cdots f_{b_1}^{y_1})^{\op}.
	\end{align*}
	
	Using the universal property we can extend $\rho_1$ and $\rho_2$ continuously to $\mathcal{A}(X,A)$ and $\mathcal{A}(Y,B)$, which are denoted as $\theta_1$ and $\theta_2$, respectively.
	By \cite[Corrollary 3.5.5]{brown88c} again we can get:
	\begin{equation*}
		\theta:\mathcal{A}(X,A)\otimes_{\min}\mathcal{A}(Y,B)\rightarrow\mathcal{U}\otimes_{\min}\mathcal{U}^{\op},~~a\otimes b\mapsto \theta_1(a)\otimes\theta_2(b)
	\end{equation*} 
	is a bounded $*$-homomorphism.
	
	Let us define 
	\[\varphi:\mathcal{A}(X,A)\otimes_{\min}\mathcal{A}(Y,B)\rightarrow\bC \]
	to be $ \varphi=\sigma\circ\theta $.
	Then by its construction, we know that $ \varphi $ is a state.  We have
	\begin{equation*}
		\varphi(e_a^x\otimes f_b^y)=\sigma(\theta_1(e_a^x)\otimes\theta_2(f_b^y))=\tau(\rho(e_a^x\otimes 1)\rho(1\otimes f_b^y))=\tau\circ\rho(e_a^x\otimes f_b^y)=p(a,b\mid x,y).
	\end{equation*}
	This shows that $ p\in\mathcal{C}_{qa}(\mathcal{G}) $.
\end{proof}

The form of Theorem \ref{Thm3.2} is similar to \cite[Theorem 3.2]{helton2019algebras}, however, the latter theorem is a necessary and sufficient condition, so we can reasonably conjecture that the converse of Theorem \ref{Thm3.2} is also true. However, in the process of proving this, we encountered difficulties and only obtained a conclusion for the following special case.

\begin{thm}\label{Thm3.1}
	Let $ \mathcal{G}=(X,Y,A,B,\lambda) $ be an imitation game, $|X|=|A|=2$ and $ p\in\mathcal{C}_{qa}(\mathcal{G}) $. Then from Proposition \ref{P2.1} we know that there exists a state $ \varphi $ on $ \mathcal{A}(X,A)\otimes_{\min}\mathcal{A}(Y,B) $ such that 
	\begin{equation*}
		\varphi(e_a^x\otimes f_b^y)=p(a,b\mid x,y).
	\end{equation*}
	The GNS construction \cite[Theorem 4.5.2]{kadison1997fundamentals} of $ \varphi $ produce a Hilbert space $ \mathcal{H} $, a unit vector $ |\psi\rangle\in \mathcal{H} $ and a $ * $-representation 
	\begin{equation}\label{pi}
		\pi:\mathcal{A}(X,A)\otimes_{\min}\mathcal{A}(Y,B)\rightarrow\mathcal{B}(\mathcal{H}) 
	\end{equation}
	such that 
	\[\varphi(a)=\langle\psi|\pi(a)|\psi\rangle,~\forall a\in\mathcal{A}(X,A)\otimes_{\min}\mathcal{A}(Y,B).\] 
	Let 
	\begin{equation}\label{UB}
		\mathfrak{A}:=\pi(\mathcal{A}(X,A)\otimes\mathbf{1}),~\mathfrak{B}:=\pi(\mathbf{1}\otimes\mathcal{A}(Y,B)).
	\end{equation}	
	It's obvious that $ \mathfrak{A} $ and $ \mathfrak{B} $ are both C$ ^* $-algebras. Suppose that $\pi\big|_{\mathcal{A}(X,A)\otimes 1}:=\pi_1$ and $\pi\big|_{1\otimes\mathcal{A}(Y,B)}:=\pi_2$ are injective. Then we have the following conclusion: 
	there exist a von Neumann algebra $ \mathcal{U} $, a $ * $-homomorphism $ \rho:\bC\langle X,A\rangle\otimes_{\alg}\bC\langle Y,B\rangle\rightarrow \mathcal{U} $ and an amenable tracial state $ \tau $ on $ \mathcal{U} $ such that 
	\begin{equation}\label{eq3.1}
		\tau\circ\rho(e_a^x\otimes f_b^y)=p(a,b\mid x,y).
	\end{equation}

	 
%
\end{thm}
\begin{proof}
	
	Since $ \varphi $ is a contractive linear functional on the separable C$^*$-algebra $ \mathcal{A}(X,A)\otimes_{\min}\mathcal{A}(Y,B) $, the Hilbert space $\mathcal{H}$ is the completion of the separable inner product space
	\begin{equation*}
	    \mathcal{A}(X,A)\otimes_{\min}\mathcal{A}(Y,B)/\{x\in\mathcal{A}(X,A)\otimes_{\min}\mathcal{A}(Y,B):\varphi(x^* x)=0\},
	\end{equation*}
	with respect to the inner product $\langle x|y\rangle:=\varphi(y^* x)$. Specifically, the quotient image of $\mathcal{C}(X,A)\otimes_{\alg}\mathcal{C}(Y,B)$ is dense in $\mathcal{H}$.

	Let 
 \begin{equation}  
\mathcal{U}=\bar{\mathfrak{A}}^{\mathrm{WOT}},~\mathcal{V}=\bar{\mathfrak{B}}^{\mathrm{WOT}} 
 \end{equation}
 be the weak-operator-topological (WOT) closure of $ \mathfrak{A},\mathfrak{B} $ respectively. 
	Then $ \mathcal{U},\mathcal{V} $ are von Neumann algebras and convex sets. 
	It's evident that $ \mathcal{U},\mathcal{V} $ commute with each other.  According to 
	Theorem 5.1.2 in  \cite{kadison1997fundamentals}, we know that  the weak-operator-topological(WOT) closure and strong-operator-topological(SOT) closure of $ \mathfrak{A} $ coincide (and equal to $ \mathcal{U} $), similarly for $ \mathcal{V} $.
	
	Since $\pi\big|_{\mathcal{A}(X,A)\otimes 1}:=\pi_1$ and $\pi\big|_{1\otimes\mathcal{A}(Y,B)}:=\pi_2$ are injective, we know that $\pi_1^{-1}$ and $\pi_2^{-1}$ are also isometric representation. Thus, by \cite[Corollary 3.5.5]{brown88c} we get a state 
	\begin{align*}
	    \phi:\mathfrak{A}\otimes_{\min}\mathfrak{B}&\rightarrow\bC\\
	    a\otimes b&\mapsto\varphi(\pi_1^{-1}(a)\otimes\pi_2^{-1}(b))
	\end{align*}
	
	We can define a state 
	\begin{eqnarray*}
	\tau:&\mathcal{U} &\rightarrow  \bC,\\
	 &a&\mapsto\langle\psi|a|\psi\rangle 
	 \end{eqnarray*}
	We show that $ \tau $ and $ \mathcal{U} $ satisfy our conclusion.  We break the arguments into smaller claims. 
	
	\begin{claim}
		$ \tau $ is a tracial state on $ \mathfrak{A} $.
	\end{claim}
	
	The proof is similar to \cite[Theorem 5.1]{lupini2020perfect}.  Let 
	\begin{equation}\label{eq3.2}
	P_a^x=\pi(e_a^x\otimes 1)\in\mathfrak{A},~x\in X,~a\in A;~Q_b^y=\pi(1\otimes f_b^y)\in\mathfrak{B},y\in Y,~b\in B.
	\end{equation}
	Then $ P_a^x, Q_b^y $ are all projections. 
	We can define a projection 
	\begin{equation}\label{eq3.3}
	{}^{x}\Pi_b^y=\sum_{a \in A, \lambda(x, y, a, b)=1} P_a^x\in\mathfrak{A}
	\end{equation} 
	for any $ x\in X,y\in Y $ and $ b\in B $. Since $ p $ is a perfect strategy, we have 
	\begin{equation*}
	\left\langle\psi\left|Q_b^y\right| \psi\right\rangle =\sum_{a \in A}\left\langle\psi\left|P_a^x Q_b^y\right| \psi\right\rangle=\sum_{a \in A, \lambda(x, y, a, b)=1}\left\langle\psi\left|P_a^x Q_b^y\right| \psi\right\rangle =\left\langle\psi\left|^x\Pi_b^y Q_b^y\right| \psi\right\rangle .
	\end{equation*}
	Since $\left(I-^x\Pi_b^y\right) Q_b^y$ is an idempotent, this shows that
	\begin{equation*}
	\|\left(I-{}^x\Pi_b^y\right) Q_b^y|\psi\rangle \|^2=\left\langle\psi\left|\left(I-^x\Pi_b^y\right) Q_b^y\right| \psi\right\rangle=0,
	\end{equation*}
	that is,
	\begin{equation}\label{eq3.6}
	{}^x\Pi_b^yQ_b^y|\psi\rangle=Q_b^y|\xi\rangle, \quad x \in X, y \in Y, b \in B .
    \end{equation}
	
	By the assumption of imitation games, for $b \neq b^{\prime}$ and $y \in Y$ there exists $x \in X$ such that 
	\begin{equation*}
	    \sum_{a \in A} \lambda(x, y, a, b) \lambda\left(x, y, a, b^{\prime}\right)=0.
	\end{equation*}
	For such a choice of $x, y, b, b^{\prime}$, we have
	\begin{equation*}
	\left\langle\psi\left|{}^x\Pi_b^y Q_{b^{\prime}}^{y}\right| \psi\right\rangle=\sum_{a \in A, \lambda(x, y, a, b)=1}\left\langle\psi\left|P_{x, a} Q_{y, b^{\prime}}\right| \psi\right\rangle=0.
	\end{equation*}
	Hence, we have
	\begin{equation}\label{eq3.7}
	    {}^x\Pi_b^y Q_{b^{\prime}}^{y}|\psi\rangle=0.
	\end{equation}
	
	Similarly to \cite{lupini2020perfect}, let
	\begin{equation}\label{PI_Y_B}
	    \Pi_b^y=\underset{x\in X}{\bigwedge}{}^x\Pi_b^y\in \mathcal{U},
	\end{equation}
	by (\ref{eq3.6}), (\ref{eq3.7}) and the definition of the imitation game, we have
	\begin{align*}
	\Pi_b^yQ_b^y|\psi\rangle&=Q_b^y|\psi\rangle,~\text{and}\\
	\Pi_b^yQ_{b^{\prime}}^y|\psi\rangle&=0
	\end{align*}
	whenever $ b^{\prime}\neq b $. Therefore, we get
	\begin{equation}\label{eq3.10}
	\Pi_b^y|\psi\rangle=\Pi_b^y\left(\sum_{b^{\prime}\in B}Q_{b^{\prime}}^{y}\right)|\psi\rangle=\Pi_b^yQ_b^y|\psi\rangle=Q_b^y|\psi\rangle,~\forall y\in Y, b\in B
	\end{equation}
	
	Similarly, we define 
	\begin{equation}\label{eq3.11}
	{}^y\Xi_a^x=\sum_{b \in B, \lambda(x, y, a, b)=1}Q_b^y\in\mathfrak{B},
	\end{equation}
	and
	\begin{equation}\label{eq3.12}
	\Xi_a^x=\underset{y\in Y}{\bigwedge}{}^y\Xi_a^x\in\mathcal{V},
	\end{equation}
	we have
	\begin{equation}\label{eq3.13}
	\Xi_a^x|\psi\rangle=P_a^x|\psi\rangle,~\forall x\in X, a\in A.
	\end{equation}
	
	Let $\mathcal{K}_{\mathcal{U}}$ (resp. $\mathcal{K}_{\mathcal{V}}$ ) be the closure of the span of $\{A|\psi\rangle: A \in \mathcal{U}\}$ (resp. $\{B|\psi\rangle: B \in \mathcal{V}\})$. Suppose that $x_1, \ldots, x_n \in X$ and $a_1, \ldots, a_n \in A$. Then we have
	\begin{equation}\label{eq3.14}
	P_{a_1}^{x_1} \cdots P_{a_n}^{x_n}|\psi\rangle=\Xi_{a_n}^{x_n} \cdots \Xi_{a_1}^{x_1}|\psi\rangle \in \mathcal{K}_{\mathcal{V}}.
	\end{equation}
	Hence, we deduce that $\mathcal{K}_{\mathcal{U}} \subseteq \mathcal{K}_{\mathcal{V}}$. Similarly, we have $\mathcal{K}_{\mathcal{V}} \subseteq \mathcal{K}_{\mathcal{U}}$. Set $\mathcal{K}:=\mathcal{K}_{\mathcal{U}}=\mathcal{K}_{\mathcal{V}}$. Clearly, $\mathcal{K}$ is invariant under the projections $P_a^x$ and $Q_b^y$. Hence, we can replace $P_a^x$ and $Q_b^y$ with their restrictions to $\mathcal{K}$. Thus, without loss of generality, we can assume that $\mathcal{K}=\mathcal{H}$. By the proof of \cite[Theorem 5.1]{lupini2020perfect}, we know that 
    \begin{equation}\label{welldef}
        \Pi_b^y\Pi_b^{y^{\prime}}=0~\text{and}~\sum_{b\in B}\Pi_b^y=1.
    \end{equation} 
	
Based on the above discussions, we can deduce that 
		\begin{align*}
	\langle\psi|P_{a_1}^{x_1}P_{a_2}^{x_2}|\psi\rangle
	=\langle\psi|P_{a_1}^{x_1}\Xi_{a_2}^{x_2}|\psi\rangle
	=\langle\psi|\Xi_{a_2}^{x_2}P_{a_1}^{x_1}|\psi\rangle
	=\langle\Xi_{a_2}^{x_1}\psi|\Xi_{a_1}^{x_1}|\psi\rangle
	=\langle P_{a_2}^{x_2}\psi|P_{a_1}^{x_1}|\psi\rangle
	=\langle\psi|P_{a_2}^{x_2}P_{a_1}^{x_1}|\psi\rangle,
	\end{align*}
	we have 
	\[	\langle\psi|P_{a_1}^{x_1}P_{a_2}^{x_2}|\psi\rangle= \langle\psi|P_{a_2}^{x_2}P_{a_1}^{x_1}|\psi\rangle.\]
		By induction and linearity, we can prove that $ \tau $ is a tracial state on $ \mathfrak{A} $. 
	\begin{claim}
	$ \tau $ is a tracial state on $ \mathcal{U} $.
	\end{claim}

	Suppose $ a,b\in \mathcal{U} $. Since we have proved that $\mathcal{H}$ is separable, according to the Corollary of \cite[Theorem 1.2.2]{arveson2012invitation}, there exist two sequences $ \{a_n\}\subset\mathfrak{A},~\{b_n\}\subset\mathfrak{A} $ such that $ a_n\rightarrow a,~b_n\rightarrow b $ under the strong operator topology. By the sequential continuity of multiplication, we have $ a_nb_n\rightarrow ab $ under the strong operator topology \cite[Problem 113]{halmos2012hilbert}.
	Specially, for any positive real number $ \varepsilon>0 $, there exists $ N\in\bN $ such that for all $ n>N $, \[ \|(ab-a_nb_n)|\psi\rangle\|<\varepsilon.\]
	Therefore, for all $ \varepsilon>0 $, there exists an  $ N\in\bN $, such that for all $ n>N $, we have
	\begin{equation}\label{eq3.15}
	\left|\langle\psi|(ab-a_nb_n)|\psi\rangle\right|\leq\||\psi\rangle\|\cdot\|(ab-a_nb_n)|\psi\rangle\|<\varepsilon,
	\end{equation}
	since we have proved that $ \tau $ is a tracial state on $ \mathfrak{A} $, i.e
	\begin{equation}\label{eq3.16}
	\langle\psi|(a_nb_n-b_na_n)|\psi\rangle=0.
	\end{equation}
	Similarly, we can prove
	\begin{equation}\label{eq3.17}
	\left|\langle\psi|(ba-b_na_n)|\psi\rangle\right|\leq\||\psi\rangle\|\cdot\|(ba-b_na_n)|\psi\rangle\|<\varepsilon.
	\end{equation}
	Combining equation (\ref{eq3.15}),(\ref{eq3.16}) and (\ref{eq3.17}) we know that  for all $ \varepsilon>0 $,
	\begin{equation}\label{eq3.18}
	\left|\langle\psi|(ab-ba)|\psi\rangle\right|\leq\left|\langle\psi|(ab-a_nb_n)|\psi\rangle\right|+\left|\langle\psi|(a_nb_n-b_na_n)|\psi\rangle\right|+\left|\langle\psi|(ba-b_na_n)|\psi\rangle\right|<2\varepsilon.
	\end{equation}
	As  $ \varepsilon $ can be arbitrarily small,  we deduce that 
	\[\langle\psi|(ab-ba)|\psi\rangle=0 \]
	for all $ a,b\in\mathcal{U} $. This shows that $ \tau $ is a tracial state on $ \mathcal{U} $.
	
	\begin{claim}
		$ \tau $ is amenable.
	\end{claim}
	According to \cite[Theorem 6.2.7]{brown88c}, we need to show that there exists a bounded linear functional 
	\[\sigma:\mathcal{U}\otimes_{\min}\mathcal{U}^{\op}\rightarrow\bC\] satisfying  \[ \sigma(a\otimes b^{\op})=\tau(ab).\] 
	We show below how to construct such $ \sigma $.
	
	
	Let
	\begin{align*}
	\varPhi:\mathcal{U}\otimes_{\min}\mathcal{V}&\rightarrow\bC\\
	a\otimes b&\mapsto\langle\psi|ab|\psi\rangle, 
	\end{align*}
	and now we prove that $\varPhi$ is bounded. 
	By \cite[Theorem 4.4.7]{Nelson209note}, it's clear that $\varPhi$ is bi-normal. Since $|X|=|A|=2$, we know $\mathcal{A}(X,A)$ is isometric to the full group C$^*$-algebra $\bC^{*}(\bZ_2^{*2})$. Since $\pi|_{\mathcal{A}(X,A)\otimes 1}$ is an isometry, we obtain that $\mathfrak{A}$ is isometric to $\mathcal{A}(X,A)\simeq \bC^{*}(\bZ_2^{*2})$. It's well known that the group $\bZ_2^{*2}$ is amenable, so $\mathfrak{A}$ is exact. By \cite[Theorem 9.3.1]{brown88c}, we know $\mathfrak{A}$ has property C, see \cite[Definition 9.2.2]{brown88c}. By definition we know $\mathfrak{A},\mathfrak{B}$ are weakly dense in $\mathcal{U},\mathcal{V}$ respectively, and $\varPhi\big|_{\mathfrak{A}\otimes_{\min}\mathfrak{B}}=\phi$. Then by \cite[Lemma 9.2.9]{brown88c}, we know $\varPhi$ is continuous, i.e bounded with respect to the minimal tensor product norm. We list this lemma here for convenience.
	\begin{Lemma}\cite[Lemma 9.2.9]{brown88c}
		Let $M$ and $N$ be von Neumann algebras with weakly dense 
		C$^*$-subalgebras $B\subset M$ and $C\subset N$, respectively. Let: $\Phi:M\otimes_{\alg}N\rightarrow\mathcal{B}(\mathcal{H})$ be 
		a bi-normal unital complete positive map.
		If $B$ has property C and $\Phi$ is continuous on $B\otimes_{\alg} C$ with respect to the minimal tensor product norm, then $\Phi$ is continuous on $M\otimes_{\alg}N$ with respect to the minimal tensor product norm. 
	\end{Lemma}
	~\\
	
	Then we define a $*$-homomophism $ s:\mathcal{U}^{\op}\rightarrow\mathcal{V} $. Firstly we define $ s $ on $ \mathfrak{A}^{\op} $ by:
	\begin{equation}\label{eq3.20}
	s((P_{a_n}^{x_n}\cdots P_{a_1}^{x_1})^{\op})=s\left((P_{a_1}^{x_1})^{\op}\bullet\cdots \bullet(P_{a_n}^{x_n})^{\op}\right)=\Xi_{a_1}^{x_1}\cdots\Xi_{a_n}^{x_n}
	\end{equation}
	It is a $*$-homomorphism from $ \mathfrak{A}^{\op} $ to $ \mathcal{V} $ obviously.
	In addition, by (\ref{eq3.14}), we know 
	\begin{equation}\label{eq*}
	s(a^{\op})|\psi\rangle=a|\psi\rangle,~\forall a\in\mathfrak{A}.
	\end{equation}
	
	Next we extend $ s $ to $ \mathcal{U}^{\op} $. Suppose $ a^{\op}\in \mathcal{U}^{\op} $, then we can find a sequence $ \{a_n^{\op}\}\subset\mathfrak{A} $ which converges to $ a^{\op} $ under the strong operator topology (coincided with the SOT on $\mathcal{U}$). For every $ |\xi\rangle\in\mathcal{H} $, we know there exists a sequence $ \{b_t\}\subset\mathcal{U} $ such that \[ |\xi\rangle=\lim\limits_{t\rightarrow\infty}c_t|\psi\rangle.\]
	That is to say, $ \forall\varepsilon>0 $, there exists $ N_1\in\bN $, for all $ t>N $ we have
	\begin{equation}\label{eq3.21}
	\||\xi\rangle-c_t|\psi\rangle\|<\varepsilon.
	\end{equation}
	Next we prove that $ \{s(a_n^{\op})|\xi\rangle\} $ is a Cauchy sequence in $ \mathcal{H} $. We have
	\begin{equation}\label{eq3.22}
	\|s(a_n^{\op}-a_m^{\op})|\xi\rangle\|\leq\|s(a_n^{\op}-a_m^{\op})b_s|\psi\rangle\|+\|s(a_n^{\op}-a_m^{\op})(|\xi\rangle-b_s\psi\rangle)\|
	\end{equation}
	Since $\{a_n^{\op}\}$ is strongly convergent, by the principle of uniform boundedness, we know $\{\|a_n^{\op}-a_m^{\op}\|\}$ has a uniform upper bound. Combined with $ s $ being bounded on $ \mathfrak{A} $ (since $s$ is a homomorphism  on $\mathfrak{A}$ and thus shrinks spectra), we can deduce that $s(a_n^{\op}-a_m^{\op})\in \mathcal{V} $ is a bounded operator, i.e \[ \|s(a_n^{\op}-a_m^{\op})\|\leq M_1,\]
	and then we can take $ t>N_1 $ such that
	\begin{equation}\label{eq3.23}
	\|s(a_n^{\op}-a_m^{\op})(|\xi\rangle-b_s\psi\rangle)\|<M_1\varepsilon.
	\end{equation}
	For $ \|s(a_n^{\op}-a_m^{\op})b_s|\psi\rangle\| $, we have
	\begin{align}\nonumber
	&\quad\|s(a_n^{\op}-a_m^{\op})b_s|\psi\rangle\|^2\\
	\nonumber&=\|b_s(a_n-a_m)|\psi\rangle\|^2\\
	\nonumber&=\langle\psi|(a_n-a_m)^*b_s^*b_s(a_n-a_m)|\psi\rangle\\
	\label{eq3.27}&=\langle\psi|b_s^*b_s(a_n-a_m)(a_n-a_m)^*|\psi\rangle\\
	\nonumber&\leq\langle\psi|b_s^*b_s|\psi\rangle^{\frac{1}{2}}\langle\psi|(a_n-a_m)(a_n-a_m)^*|\psi\rangle^{\frac{1}{2}}\\
	\label{eq3.29}&=\langle\psi|b_s^*b_s|\psi\rangle^{\frac{1}{2}}\langle\psi|(a_n-a_m)^*(a_n-a_m)|\psi\rangle^{\frac{1}{2}},
	\end{align}
	where (\ref{eq3.27}) and (\ref{eq3.29}) are obtained since $ \tau(\cdot)=\langle\psi|\cdot|\psi\rangle $ is a tracial state on $ \mathfrak{A} $.
	Notice that when $ s\rightarrow\infty $, 
	\[\langle\psi|b_s^*b_s|\psi\rangle^{\frac{1}{2}}=\|b_s|\psi\rangle\|\rightarrow\||\xi\rangle\|.\]
	Hence, we can take a sufficiently large $ s $ such that 
	\[\langle\psi|b_s^*b_s|\psi\rangle^{\frac{1}{2}}<2\||\xi\rangle\|.\]
	Since \[ \langle\psi|(a_n-a_m)^*(a_n-a_m)|\psi\rangle^{\frac{1}{2}}=\|(a_n-a_m)|\psi\rangle\|,\]
	and $ \{a_n^{\op}\} $  converges to $ a^{op} $ under the strong operator topology, and the elements and topology of $ \mathcal{U}^{\op} $ and $ \mathcal{U} $ coincide, 
	we know $ \{a_n|\psi\rangle\} $ is a Cauchy sequence. That is to say, $ \forall\varepsilon>0 $, there exists $ N_2\in\bN $ such that $ \forall n,m>N_2 $,
	\begin{equation*}
	\langle\psi|(a_n-a_m)^*(a_n-a_m)|\psi\rangle^{\frac{1}{2}}<\varepsilon.
	\end{equation*}
	Without loss of generality,  we can assume $ \varepsilon<1 $. Thus there exists $ M_2=\sqrt{2\||\xi\rangle\|}>0 $ such that 
	\begin{equation}\label{eq3.31}
	\|s(a_n^{\op}-a_m^{\op})b_s|\psi\rangle\|<M_2\varepsilon
	\end{equation}
	
	After combining (\ref{eq3.22}), (\ref{eq3.23}) and (\ref{eq3.31}), 
	we have proved that $ \{s(a_n^{\op})|\xi\rangle\} $ is a Cauchy sequence. 
	So we can define $ s(a^{\op})|\xi\rangle $  to be the limit of $ \{s(a_n^{\op})|\xi\rangle\} $. 
	It's not hard to verify that this definition of $ s:\mathcal{U}^{\op}\rightarrow\mathcal{V} $ is well defined. 
	
	The remaining work is to prove that $ s $ is a $*$-homomorphism. 
	Suppose $ a^{\op},b^{\op}\in \mathcal{U}^{\op} $, and there exist two sequences $ \{a_n^{\op}\},\{b_n^{\op}\}\subset\mathfrak{A}^{\op} $ that strongly converge to $ a^{\op}, b^{\op} $, respectively. 
	By the sequential continuity of multiplication, we have $ a_n^{\op}\bullet b_n^{\op}=(b_na_n)^{op} $ strongly converges to $ a^{\op}\bullet b^{\op}=(ba)^{\op} $.
	According to the definition of $ s $ on $ \mathcal{U}^{\op} $, 
	for all $ |\xi\rangle\in\mathcal{H} $,  we have
	\begin{align}
	\label{eq3.33}s(a^{\op}\bullet b^{\op})|\xi\rangle&=\lim\limits_{n\rightarrow\infty}s(a_n^{\op}\bullet b_n^{\op})|\xi\rangle\\
	\label{eq3.34}&=\lim\limits_{n\rightarrow\infty}s(a_n^{\op})s(b_n^{\op})|\xi\rangle\\
	\nonumber&=\lim\limits_{n\rightarrow\infty}s(a_n^{\op})s(b^{\op})|\xi\rangle\\
	\nonumber&=s(a^{\op})s(b^{\op})|\xi\rangle,
	\end{align}
	where (\ref{eq3.33}) is obtained according to  the definition of $ s $, and (\ref{eq3.34}) is true since   $ s $ is a $*$-homomorphism on $ \mathfrak{A} $. Thus $ s $ is a $*$-homomorphism from $ \mathcal{U}^{\op} $ to $ \mathcal{V} $, and it's also bounded by \cite[Theorem 1.3.2]{arveson2012invitation}. What's more, by continuity (i.e., boundedness) of $ s $, we know (\ref{eq*}) holds for all $ a^{\op}\in\mathcal{U}^{\op} $.
	
	Now we define $ \sigma:\mathcal{U}\otimes_{\min}\mathcal{U}^{\op}\rightarrow\bC $ to be
	\begin{equation}\label{eq3.37}
	\sigma=\varPhi\circ(\mathrm{id}\otimes s),
	\end{equation}
	and we have proved that $ \sigma $ is bounded. Moreover, for all $ a\in \mathcal{U} $ and $ b\in\mathcal{U}^{\op} $, we have
	\begin{equation}\label{eq3.38}
	\sigma(a\otimes b^{\op})=\varPhi(a\otimes s(b^{\op}))=\langle\psi|a\cdot s(b^{\op})|\psi\rangle=\langle\psi|a\cdot b|\psi\rangle=\tau(ab).
	\end{equation}
	This shows that $ \tau $ is amenable.
	\\~
	
	Finally, let 
	\begin{align*}
	\rho:\bC\langle X,A\rangle\otimes_{\alg}\bC\langle Y,B\rangle&\rightarrow\mathcal{U}\\
	e_a^x\otimes f_b^y&\mapsto P_a^x\cdot\Pi_b^y.
	\end{align*}
	By (\ref{welldef}) we know $\rho$ is well defined. It's easy to see that
	\begin{equation}\label{eq3.39}
	\tau\circ\rho(e_a^x\otimes f_b^y)=\langle\psi|P_a^x\Pi_b^y|\psi\rangle=\langle\psi|P_a^xQ_b^y|\psi\rangle=\langle\psi|\pi(e_a^x\otimes f_b^y)|\psi\rangle=\varphi(e_a^x\otimes f_b^y)=p(a,b\mid x,y).
	\end{equation}
	Thus we have proved the theorem.
\end{proof}

\begin{remark}
 The von Neumann algebra  
$\mathcal{U}=\bar{\pi(\mathcal{A}(X,A)\otimes\mathbf{1})}^{\mathrm{WOT}},$ that we use in Theorem \ref{Thm3.1} (the *-homomorphism $\pi$ is given in (\ref{pi})) plays the same role of the small algebra $\mathcal{A}(X,A)$ in  Proposition \ref{Mirror}. We use von Neumann algebra instead of C$^*$-algebra, because $\Pi_b^y$ in (\ref{PI_Y_B}) is the intersection of projections in C$^*$-algebra $\mathfrak{A}=\pi(\mathcal{A}(X,A)\otimes\mathbf{1})$,  thus $\Pi_b^y$ may not necessarily belong to $\mathfrak{A}$. 
\end{remark}

\section{Our Problems}

In Theorem \ref{Thm3.2} and \ref{Thm3.1}, we try to generalize the result of \cite{lupini2020perfect}, give a characterization of perfect approximate strategies for imitation games using the amenable tracial state.
Our conjectured main theorem should be of the following form.

\begin{prop}[Not Proved]\label{P_conj}
	Let $ \mathcal{G}=(X,Y,A,B,\lambda) $ be an imitation game, and $ p\in\mathcal{C}_{ns}(\mathcal{G}) $ be a perfect non-signalling correlation. The following statements are equivalent:
	\begin{enumerate}[(i)]
		\item $p \in \mathcal{C}_{qa}$;
		\item there exist a von Neumann algebra $\mathcal{U}$ with an amenable tracial state $\tau$ on it, and a $^*$-homomorphism $\rho$ from the game algebra $\mathbb{C}\langle X, A \rangle \otimes_{\alg} \mathbb{C}\langle Y, B \rangle$ to $\mathcal{U}$, such that
		\[
		\tau \circ (e_a^x \otimes f_b^y) = p(a, b \mid x, y).
		\]
	\end{enumerate}
\end{prop}

Theorem \ref{Thm3.2} has proved the direction (ii)$\Rightarrow$(i).
For (i)$\Rightarrow$(ii), we think the von Neumann algebra $\mathcal{U}$ should be constructed as the closure of $\pi(\mathcal{A}(X, A) \otimes 1)$ in the weak operator topology. 
To prove that $\tau$ is amenable, we need to demonstrate that the mapping $\mathcal{U} \otimes_{\min} \mathcal{U}^{\op} \rightarrow \mathbb{C},~a \otimes b \mapsto \tau(ab)$ is continuous with respect to the minimal tensor norm. 
To this end, we attempt to prove the continuity of $\tau$ via the following path:
\begin{equation*}
	\mathcal{U} \otimes_{\min} \mathcal{U}^{op} \xrightarrow{\mathrm{id} \otimes s} \mathcal{U} \otimes_{\min} \mathcal{V} \xrightarrow{\alpha}{\mathcal{B}(\mathcal{H})}\xrightarrow{a\otimes b\mapsto\langle \psi | a \cdot b | \psi \rangle} \mathbb{C},
\end{equation*}
where $\mathcal{H}$ is from the GNS construction of $\varphi$.

In the proof of Theorem \ref{Thm3.1}, we have already obtained the continuity of $s$. But the mapping 
\begin{equation*}
	\alpha: \mathcal{U} \otimes_{\min} \mathcal{V} \rightarrow\mathcal{B}(\mathcal{H}),~a\otimes b\mapsto a\cdot b
\end{equation*}
may not be continuous with respect to the minimal tensor norm. We note that there exists the following counterexample.
\begin{Example}\cite[Exercise 3.6.3]{brown88c}
	Let $\Gamma$ be a discrete group and let 
	$\lambda \times \rho: C_{\lambda}^{*}(\Gamma) \odot C_{\rho}^{*}(\Gamma) \rightarrow \mathbb{B}\left(\ell^{2}(\Gamma)\right)$
	be the product of the left and right regular representations. Then $\Gamma$  is amenable if and only if  $\lambda \times \rho$  is continuous with respect to the minimal tensor product norm.
\end{Example}

Thus, if we take $\Gamma=F_2$ as the free group generated by two elements, since $F_2$ is not amenable, we know $\lambda \times \rho$ is not continuous with respect to the minimal tensor product norm. 

However, this counterexample does not directly disprove Proposition \ref{P_conj}, because $\mathcal{U}$ and $\mathcal{V}$ have more structure. \cite[Lemma 9.2.9]{brown88c} is a sufficient condition to get min-continuity, and we don't know whether $\alpha$ is continuous with respect to the minimal tensor product norm for the general case. Therefore, we don't know whether Proposition \ref{P_conj} holds for general imitation games, or it need more conditions to be true.  
What's more, Theorem \ref{independent} is independent of Proposition \ref{P_conj}. It is more refined than Proposition \ref{P2.1}, but we do not know if it is helpful for proving Proposition \ref{P_conj}.

\bibliographystyle{unsrt}
\bibliography{LYZ2024ref.bib}
\end{document}